\documentclass[12pt,leqno]{article}

\usepackage{amsmath,amssymb,amsthm}

\makeatletter

\newtheorem{Theorem}{Theorem}

\newtheorem{Lemma}[Theorem]{Lemma}
\newtheorem{Corollary}[Theorem]{Corollary}

\theoremstyle{definition}
\newtheorem{Example}[Theorem]{Example}

\theoremstyle{remark}
\newtheorem{Remark}[Theorem]{Remark}

\def\({{\rm (}}
\def\){{\rm )}}

\let\Mathrm\operator@font

\newcommand{\fp}{\ensuremath{\mathfrak p}}

\def\standop#1{\mathop{\Mathrm #1}\nolimits}
\def\difstop#1#2{\expandafter\def\csname #1\endcsname{\standop{#2}}}
\def\defstop#1{\difstop{#1}{#1}}

\defstop{codim}
\defstop{id}
\defstop{Sym}
\defstop{Hom}
\defstop{End}
\defstop{rank}
\defstop{Spec}

\defstop{Ker}
\defstop{Min}
\difstop{Image}{Im}
\difstop{tdeg}{trans.deg}

\defstop{length}
\defstop{supp}
\defstop{Proj}
\defstop{Flat}
\defstop{Alt}
\defstop{Mat}
\defstop{Ext}
\defstop{Tor}
\defstop{Cobar}
\defstop{Coker}
\defstop{Max}
\difstop{height}{ht}
\difstop{Image}{Im}

\def\op{^{\standop{op}}}
\newcommand{\bs}[1]{\boldsymbol{#1}}

\def\specialarrow#1{\setbox\z@=\hbox{$\m@th
 \mathop{\vphantom{\rightarrow}}\limits^{\hspace{.5ex}{#1}\hspace
{.8ex}}$}\mathrel{\ifdim\wd\z@<1.2em\dimen\tw@
1.2em\else\dimen\tw@\wd\z@\fi\copy\z@\kern-\wd\z@\hbox to\dimen\tw@
{\rightarrowfill}}}

\def\sdarrow#1{\downarrow\hbox to 0pt{\scriptsize$#1$\hss}}
\def\suarrow#1{\uparrow\hbox to 0pt{\scriptsize$#1$\hss}}
\def\ssearrow#1{\searrow\hbox to 0pt{\scriptsize$#1$\hss}}

\let\indlim\varinjlim

\def\section{\@startsection{section}{1}{\z@ }%
{-3.5ex plus -1ex minus -.2ex}{2.3ex plus .2ex}{\bf }}

\long\def\refname{\par\kern -3ex
\begin{center}\rm R\sc{eferences}\end{center}\par\kern 
-2ex}

\def\@seccntformat#1{\csname the#1\endcsname.\quad}

\def\@@@sect#1#2#3#4#5#6[#7]#8{%
   \ifnum #2>\c@secnumdepth 
      \def \@svsec {}\else \refstepcounter {#1}%
      \def\@svsec{}
   \fi 
   \@tempskipa #5\relax 
   \ifdim \@tempskipa >\z@ 
     \begingroup #6\relax \@hangfrom {\hskip #3\relax 
     \@svsec}{\interlinepenalty \@M #8\par }\endgroup 
     \csname #1mark\endcsname {#7}
   \else 
   \def \@svsechd {#6\hskip #3\@svsec #8\csname #1mark\endcsname {#7}}
   \fi \@xsect {#5}}

\def\@@@startsection#1#2#3#4#5#6{%
 \if@noskipsec \leavevmode \fi \par \@tempskipa #4\relax \@afterindenttrue 
 \ifdim \@tempskipa <\z@ \@tempskipa -\@tempskipa \@afterindentfalse 
 \fi \if@nobreak \everypar {}\else \addpenalty {\@secpenalty }\addvspace 
  {\@tempskipa }\fi \@ifstar {\@ssect {#3}{#4}{#5}{#6}}{\@dblarg 
  {\@@@sect {#1}{#2}{#3}{#4}{#5}{#6}}}}

\def\theparagraph{\thesection.\arabic{paragraph}}
\def\aparagraph{\@@@startsection{paragraph}{2}{\z@ }%
              {1.75ex plus .2ex minus .15ex}{-1em}{\bf(\theparagraph) } }
\def\paragraph{\@@@startsection{paragraph}{2}{\z@ }%
              {1.75ex plus .2ex minus .15ex}{-1em}{}{\bf(\theparagraph)} }

\let\c@Theorem\c@paragraph

\title{Acyclicity of complexes of flat modules}

\author{M{\sc itsuyasu} H{\sc ashimoto}}
\date{\normalsize
Graduate School of Mathematics, Nagoya University\\
Chikusa-ku,  Nagoya 464--8602 JAPAN\\
{\small \tt hasimoto@math.nagoya-u.ac.jp}\\
~\\
Dedicated to Professor Masayoshi Nagata on his eightieth birthday
}

\begin{document}
\maketitle
\footnote[0]{2000 \textit{Mathematics Subject Classification}. 
    Primary 13C11; Secondary 13C10.}%
\begin{abstract}
Let $R$ be a noetherian commutative ring, and 
\[
\Bbb F: \cdots \rightarrow F_2\rightarrow F_1\rightarrow F_0\rightarrow 0
\]
a complex of flat $R$-modules.
We prove that
if $\kappa(\frak p)\otimes_R\Bbb F$ is acyclic for every $\frak p\in\Spec R$,
then $\Bbb F$ is acyclic, and $H_0(\Bbb F)$ is $R$-flat.
It follows that if 
$\Bbb F$ is a (possibly unbounded)
complex of flat $R$-modules and
$\kappa(\frak p)\otimes_R \Bbb F$ is exact for every $\frak p\in\Spec R$,
then $\Bbb G\otimes_R^\bullet\Bbb F$ is exact for every $R$-complex $\Bbb G$.
If, moreover, $\Bbb F$ is a complex of projective $R$-modules,
then it is null-homotopic (follows from Neeman's theorem).
\end{abstract}

\section{Introduction}

Throughout this paper, $R$ denotes a noetherian commutative ring.
The symbol $\otimes$ without any subscript means $\otimes_R$.
For $\frak p\in\Spec R$, let $-(\frak p)$ denote the functor 
$\kappa(\frak p)\otimes-$, where $\kappa(\frak p)$ is the field
$R_{\frak p}/\frak p R_{\frak p}$.
An $R$-complex of the form 
\[
\Bbb F: \cdots \xrightarrow{d_2} F_1\xrightarrow{d_1}F_0\rightarrow 0
\]
is said to be {\em acyclic} if $H_i(\Bbb F)=0$ for every $i>0$.

In this paper, we prove:
\begin{Theorem}\label{main.thm}
Let 
\[
\Bbb F: \cdots \xrightarrow{d_2} F_1\xrightarrow{d_1}F_0\rightarrow 0
\]
be a complex of $R$-flat modules.
If $\Bbb F(\frak p)$ is acyclic for every $\frak p\in\Spec R$, then
$\Bbb F$ is acyclic, and $H_0(\Bbb F)$ is $R$-flat.
In particular, $M\otimes\Bbb F$ is acyclic for every $R$-module $M$.
\end{Theorem}

It has been known 
that, for an $R$-linear map of 
$R$-flat modules $\varphi:F_1\rightarrow F_0$, if $\varphi(\frak p)$
is injective for every  
$\frak p\in\Spec R$, then $\varphi$ is injective and $\Coker
\varphi$ is $R$-flat (see \cite[Lemma~4.2]{Enochs}, 
\cite[Lemma~I.2.1.4]{Hashimoto} and 
Corollary~\ref{map.thm}).
This is the special case of 
the theorem where $F_i=0$ for every $i\geq 2$.
The new proof of the theorem is simpler than the proofs of the special case
in \cite{Enochs} and \cite{Hashimoto}.

By the theorem, it follows immediately that if $\Bbb F$ is an (unbounded)
complex of $R$-flat modules 
and $\Bbb F(\frak p)$ is exact for every $\frak p\in\Spec R$, then
$\Bbb F$ is {\em $K$-flat} (to be defined below) and exact.
Combining this and Neeman's result, we can also prove that
an (unbounded) complex $\Bbb P$ of $R$-projective modules 
is null-homotopic if 
$\Bbb P(\frak p)$ is exact for every $\frak p\in\Spec R$.

The author is grateful to H. Brenner for a valuable discussion.
Special thanks are also due to A. Neeman for sending his
preprint \cite{Neeman} to the author.
The author thanks the referee for valuable comments.

\section{Main results}

We give a proof of Theorem~\ref{main.thm}.

\begin{proof}[Proof of Theorem~\ref{main.thm}] 
It suffices to prove that $R/I\otimes \Bbb F$ is acyclic for every
ideal $I$ of $R$.
Indeed, if so, then considering the case that $I=0$, we have that $\Bbb F$
is acyclic so that it is a flat resolution of $H_0(\Bbb F)$.
Since $R/I\otimes \Bbb F$ is acyclic for every ideal $I$, we have that
$\Tor_i^R(R/I,H_0(\Bbb F))=0$ for every $i>0$.
Thus $H_0(\Bbb F)$ is $R$-flat.
So $\Tor_i^R(M,H_0(\Bbb F))=0$ for every $i>0$, and the last assertion 
of the theorem follows.

Assume the contrary, and let $I$ be maximal among the ideals $J$ such that
$R/J\otimes \Bbb F$ is not acyclic.
Then replacing $R$ by $R/I$ and $\Bbb F$ by $R/I\otimes \Bbb F$, we may
assume that $R/I\otimes\Bbb F$ is acyclic for every nonzero ideal $I$ of
$R$, but $\Bbb F$ itself is not acyclic.

Assume that $R$ is not a domain.
There exists a filtration
\[
0=M_0\subset M_1\subset\cdots\subset M_r=R
\]
such that for each $i$, $M_i/M_{i-1}\cong R/\fp_i$ for some $\fp_i\in
\Spec R$.
Since each $\fp_i$ is a nonzero ideal, $R/\fp_i\otimes\Bbb F$ is acyclic.
So $M_i\otimes\Bbb F$ is acyclic for every $i$.
In particular, $\Bbb F\cong M_r\otimes\Bbb F$ is acyclic, and this is
a contradiction.
So $R$ must be a domain.

For each $x\in R\setminus0$, there is an exact sequence
\[
0\rightarrow \Bbb F\xrightarrow x \Bbb F\rightarrow R/Rx\otimes \Bbb F
\rightarrow 0.
\]
Since $R/Rx\otimes\Bbb F$ is acyclic, we have that
$x: H_i(\Bbb F)\rightarrow H_i(\Bbb F)$ is an isomorphism for every $i>0$.
In particular, $H_i(\Bbb F)$ is a $K$-vector space, where $K=\kappa(0)$ is the
field of fractions of $R$.
So 
\[
H_i(\Bbb F)\cong K\otimes H_i(\Bbb F)\cong H_i(K\otimes\Bbb F)
=H_i(\Bbb F(0))=0\qquad (i>0),
\]
and this is a contradiction.
\end{proof}

Let $A$ be a ring.
A complex $\Bbb F$ of left $A$-modules 
is said to be {\em $K$-flat} if the tensor product
$\Bbb G\otimes_A^\bullet \Bbb F$ is exact for every exact complex $\Bbb G$ of 
right $A$-modules, see \cite[Definition~5.1]{Spaltenstein}.

For a chain complex
\[
\Bbb H:\cdots \rightarrow H_{i+1}\xrightarrow{d_{i+1}}
H_i\xrightarrow{d_i}H_{i-1}\rightarrow\cdots
\]
of left or right $A$-modules, we denote the complex
\[
\cdots \rightarrow H_{i+1}\rightarrow \Ker d_i\rightarrow 0
\]
by $\tau_{\geq i}\Bbb H$ or $\tau^{\leq -i}\Bbb H$.
Since $\Bbb G\cong \indlim \tau^{\leq n}\Bbb G$, $\Bbb F$ is $K$-flat
if and only if $\Bbb G\otimes_A^\bullet \Bbb F$ is exact for every
exact complex $\Bbb G$ of right $A$-modules 
bounded above (i.e., $\Bbb G_{-i}=\Bbb G^{i}=0$ for $i\gg0$).
A complex $\Bbb F$ of flat left $A$-modules 
is $K$-flat if it is bounded above,
as can be seen easily from the spectral sequence argument.
A null-homotopic complex $\Bbb F$ is $K$-flat, since 
$\Bbb G\otimes_A^\bullet \Bbb F$ is null-homotopic for every complex $\Bbb G$.

\begin{Lemma}%
%[\citinfo]%
\label{K-flat.thm}
Let $A$ be a ring, and
\[
\Bbb F:\cdots\rightarrow F_{i+1}\xrightarrow{d_{i+1}} F_i
\xrightarrow{d_i}F_{i-1}\rightarrow\cdots
\]
a complex of flat left $A$-modules.
Then the following are equivalent.
\begin{description}
\item[\rm(1)] $M\otimes_A\Bbb F$ is exact for every right $A$-module $M$.
\item[\rm(2)] $\Bbb F$ is exact, and $\Image d_i$ is flat for every $i$.
\item[\rm(3)] For every complex $\Bbb G$ of right $A$-modules, 
$\Bbb G\otimes_A^\bullet \Bbb F$ is exact.
\item[\rm(4)] $\Bbb F$ is $K$-flat and exact.
\end{description}
\end{Lemma}

\begin{proof} (1) $\Rightarrow$ (2).
Obviously, $\Bbb F\cong A\otimes_A \Bbb F$ is exact.
Thus
\[
\Bbb F':\cdots\rightarrow F_{i+1}\rightarrow F_{i}\rightarrow 0
\]
is a flat resolution of $\Image d_i$, where $F_{n+i}$ has the 
homological degree $n$ in $\Bbb F'$.
For every $i\in\Bbb Z$, 
\[
\Tor_1^A(M,\Image d_i)\cong H_1(M\otimes_A\Bbb F')\cong H_{i+1}(M\otimes_A
\Bbb F)=0
\]
for every right $A$-module $M$.
Thus $\Image d_i$ is $A$-flat.

(2) $\Rightarrow$ (1).
For every $i\in\Bbb Z$,
\[
H_{i+1}(M\otimes_A\Bbb F)\cong H_1(M\otimes_A\Bbb F')\cong
\Tor_1^A(M,\Image d_i)=0,
\]
where $\Bbb F'$ is as above.

(1), (2) $\Rightarrow$ (3). Since $\Bbb G\cong\indlim\tau^{\leq n}\Bbb G$,
we may assume that $\Bbb G$ is bounded above.
Since $\Bbb F\cong \indlim \tau^{\leq n}\Bbb F$ and 
$\tau^{\leq n}\Bbb F$ satisfies (2) (and hence (1)),
we may assume that $\Bbb F$ is also bounded above.
Then by an easy spectral sequence argument, $\Bbb G\otimes_A^\bullet\Bbb F$ is
exact.

(3) $\Rightarrow$ (4) is trivial.

(4) $\Rightarrow$ (1). Let $\Bbb P$ be a projective resolution of $M$.
Since $\Bbb P$ is a bounded above complex of flat left $A\op$-modules and
$\Bbb F$ is an exact complex of right $A\op$-modules, 
$\Bbb P\otimes_A^\bullet \Bbb F$ is exact.
Let $\Bbb Q$ be the mapping cone of $\Bbb P\rightarrow M$.
Then $\Bbb Q\otimes_A^\bullet\Bbb F$ is also exact, since $\Bbb Q$ is exact and
$\Bbb F$ is $K$-flat.
By the exact sequence of homology groups
\[
H_i(\Bbb P\otimes_A^\bullet\Bbb F)\rightarrow H_i(M\otimes_A\Bbb F)
\rightarrow H_i(\Bbb Q\otimes_A^\bullet\Bbb F),
\]
we have that $M\otimes_A \Bbb F$ is also exact.
\end{proof}

In \cite[Proposition~5.7]{Spaltenstein}, (4) $\Rightarrow$ (3) above is proved 
essentially.

\begin{Corollary}\label{unbounded.thm}
Let
\[
\Bbb F: \cdots \rightarrow F_{n+1}\xrightarrow{d_{n+1}}F_n\xrightarrow
{d_n}F_{n-1}\rightarrow\cdots
\]
be a \(possibly unbounded\) complex of flat 
$R$-modules.
If $\Bbb F(\frak p)$ is exact for every $\frak p\in\Spec R$, then 
$\Bbb F$ is $K$-flat and exact.
\end{Corollary}

\begin{proof} By Lemma~\ref{K-flat.thm}, 
it suffices to show that for every $n\in \Bbb Z$ and every $R$-module
$M$, $H_n(M\otimes\Bbb F)=0$.
But this is trivial by Theorem~\ref{main.thm} applied to the complex
\[
\cdots \rightarrow F_{n+1}\xrightarrow{d_{n+1}}F_n\xrightarrow
{d_n}F_{n-1}\rightarrow 0.
\]
\end{proof}

The following was proved by A. Neeman \cite[Corollary~6.10]{Neeman}.

\begin{Theorem}\label{Neeman.thm}
Let $A$ be a ring, and $\Bbb P$ a complex of projective left $A$-modules.
If $\Bbb P$ is $K$-flat and exact, then $\Bbb P$ is null-homotopic.
\end{Theorem}

By Corollary~\ref{unbounded.thm} and Theorem~\ref{Neeman.thm}, we have

\begin{Corollary}
Let $\Bbb P$ be a complex of $R$-projective modules.
If $\Bbb P(\frak p)$ is exact for every $\frak p\in\Spec R$, then
$\Bbb P$ is null-homotopic.
\end{Corollary}

The following also follows.

\def\citinfo{\cite[Lemma~4.2]{Enochs}, \cite[Lemma~I.2.1.4]{Hashimoto}}
\begin{Corollary}[\citinfo]\label{map.thm}
Let $\varphi: F_1\rightarrow F_0$ be an $R$-linear map between $R$-flat
modules.
Then the following are equivalent.
\begin{description}
\item[1] $\varphi$ is injective and $\Coker\varphi$ is $R$-flat.
\item[2] $\varphi$ is pure.
\item[3] $\varphi(\frak p)$ is injective for every $\frak p\in\Spec R$.
\end{description}
\end{Corollary}

\begin{proof} {\bf 1$\Rightarrow$2$\Rightarrow$3} is obvious.
{\bf 3$\Rightarrow$1} is a special case of Theorem~\ref{main.thm}.
\qed

\def\citinfo{\cite[Corollary~I.2.1.6]{Hashimoto}}
\begin{Corollary}[\citinfo]\label{zero.thm}
Let $F$ be a flat $R$-module.
If $F(\frak p)=0$ for every $\frak p\in\Spec R$, then $F=0$.
\end{Corollary}

\proof Consider the zero map $F\rightarrow 0$, and apply
Corollary~\ref{map.thm}.
We have that this map is injective, and hence $F=0$.
\end{proof}

If $M$ is a finitely generated $R$-module and $M(\frak m)=0$ for every
maximal ideal $\frak m$ of $R$, then $M=0$.
This is a consequence of Nakayama's lemma.

\begin{Corollary}\label{isom.thm}
Let $\varphi:F_1\rightarrow F_0$ be an $R$-linear map between $R$-flat
modules.
If $\varphi(\frak p)$ is an isomorphism for every $\frak p\in\Spec R$, then
$\varphi$ is an isomorphism.
\end{Corollary}

\begin{proof} By Corollary~\ref{map.thm}, $\varphi$ is injective and $C:=
\Coker\varphi$ is $R$-flat.
Since $C(\frak p)\cong \Coker(\varphi(\frak p))=0$ for every 
$\frak p\in\Spec R$,
we have that $C=0$ by Corollary~\ref{zero.thm}.
\end{proof}

\begin{Corollary}\label{tor.thm}
Let $M$ be an $R$-module.
If $\Tor_i^R(\kappa(\frak p),M)=0$ for every $i>0$ and every prime
ideal $\fp\in\Spec R$, then $M$ is 
$R$-flat.
If $\Tor_i^R(\kappa(\frak p),M)=0$ for every $i\geq 0$ and every
prime ideal $\fp\in\Spec R$, then $M=0$.
\end{Corollary} 

\begin{proof} For the first assertion, 
Let $\Bbb F$ be a projective resolution of $M$, and
apply Theorem~\ref{main.thm}.
The second assertion follows from the first assertion and
Corollary~\ref{zero.thm}.
\end{proof}

\begin{Corollary}\label{ext.thm}
Let $M$ be an $R$-module.
If $\Ext^i_R(M,\kappa(\frak p))=0$ for every $i>0$ \(resp.\ $i\geq 0$\) 
and every prime ideal $\fp\in\Spec R$, 
then $M$ is $R$-flat \(resp.\ $M=0$\).
\end{Corollary}

\begin{proof} This is trivial by Corollary~\ref{tor.thm} and the fact
\[
\Ext^i_R(M,\kappa(\frak p))\cong \Hom_{\kappa(\frak p)}(\Tor_i^R(
\kappa(\frak p),M),\kappa(\frak p)).
\]
\end{proof}

\section{Some examples}

\begin{Example}
There is an acyclic projective complex
\[
\Bbb P:\cdots \rightarrow P_1\rightarrow P_0\rightarrow 0
\]
over a noetherian commutative ring $R$ such that $H_0(\Bbb P)$ is
$R$-flat and $h_0(\frak p):=\dim_{\kappa(\frak p)}H_0(\Bbb P(\frak p))$ is
finite and constant, but $H_0(\Bbb P)$ is neither $R$-finite nor 
$R$-projective.
\end{Example}

\begin{proof} Set $R=\Bbb Z$, $M=\sum_{p \text{ prime}}(1/p)\Bbb Z
\subset \Bbb Q$, 
and $\Bbb P$ to be a
projective resolution of $M$.
Then $M$ is $R$-torsion free, and is $R$-flat.
Since $M_{(p)}=(1/p)\Bbb Z_{(p)}$, $h_0(\frak p)=1$ for every $\frak p\in
\Spec \Bbb Z$.
A finitely generated nonzero 
$\Bbb Z$-submodule of $\Bbb Q$ must be
rank-one free, but $M$ is not a cyclic module, and is not rank-one free.
This shows that $M$ is not $R$-finite.
As $R$ is a principal ideal domain, every $R$-projective module is free.
If $M$ is projective, then it is free of rank $h_0((0))=1$.
But $M$ is not finitely generated, so $M$ is not projective.
\end{proof}

\vskip 1ex

\begin{Remark}
Let $(R,\frak m)$ be a noetherian {\em local} ring, $F$ a flat $R$-module,
and $c$ a non-negative integer.
If $\dim_{\kappa(\frak p)}F(\frak p)=c$ for every $\frak p\in\Spec R$, then
$F\cong R^c$, see \cite[Corollary~III.2.1.10]{Hashimoto}.
\end{Remark}

\vskip 1ex

\begin{Remark} Let
\[
\Bbb P: 0\rightarrow P^0\xrightarrow{d^0} 
P^1\xrightarrow{d^1} P^2\xrightarrow{d^2} \cdots
\]
be a complex of $R$-flat modules such that $P^0$ is $R$-projective.
Assume that 
$\Bbb P(\frak p)$ is acyclic (i.e., $H^i(\Bbb P(\frak p))=0$ for every $i>0$) 
and $h^0_{\Bbb P}(\frak p):=\dim_{\kappa(\frak p)}H^0(\Bbb P(\frak p))$ is 
finite for every $\frak p\in\Spec R$.
If $h^0_{\Bbb P}$ is a locally constant function on $\Spec R$, then
$H^0(\Bbb P)$ is 
$R$-finite $R$-projective, $H^i(M\otimes\Bbb P)=0$ ($i>0$), and 
the canonical map $M\otimes H^0(\Bbb P)
\rightarrow H^0(M\otimes \Bbb P)$ is an isomorphism for every $R$-module $M$,
see \cite[Proposition~III.2.1.14]{Hashimoto}.
If, moreover, $\Bbb P$ is a complex of $R$-projective modules, 
then $\Image d^i$ is
$R$-projective for every $i\geq 0$, 
as can be seen easily from Theorem~\ref{Neeman.thm}.
\end{Remark}

\begin{Example} 
Let $M$ be an $R$-module.
Even if $M(\frak p)=0$ for every $\frak p\in\Spec R$, $M$ may not be zero.
Even if $\Tor^R_1(\kappa(\frak p),M)=0$ for every $\frak p\in\Spec R$, 
$M$ may not be $R$-flat.

Indeed, 
let $(R,\frak m,k)$ be a $d$-dimensional regular local 
ring, and $E$ the injective hull of $k$.
Then
\[
\Tor^R_i(\kappa(\frak p),E)\cong 
\left\{
\begin{array}{ll}
k & \text{for } i=d \text{ and } \frak p=\frak m \\
0 & \text{otherwise}
\end{array}
\right..
\]
$E$ is not $R$-flat unless $d=0$.
\end{Example}

\begin{proof} 
Since $\supp E=\{\frak m\}$, $\Tor^R_i(\kappa(\frak p),E)=0$ unless
$\frak p=\frak m$.

Let $\bs x=(x_1,\ldots,x_d)$ be a regular system of parameters of $R$,
and $\Bbb K$ the Koszul complex $K(\bs x;R)$, which is a minimal free
resolution of $k$.
Note that $\Bbb K$ is self-dual.
That is, $\Bbb K^*\cong \Bbb K[-d]$,
where $\Bbb K^*=\Hom_R^\bullet(\Bbb K,R)$, and $\Bbb K[-d]^n=\Bbb K^{n-d}$.
So
\begin{multline*}
\Tor_i^R(k,E)\cong H^{-i}(\Bbb K\otimes E)
\cong H^{-i}(\Bbb K^{**}\otimes E)
\cong H^{-i}(\Hom^\bullet_R(\Bbb K[-d],E))\\
\cong H^{-i}(\Hom^\bullet_R(k[-d],E))
\cong
\left\{
\begin{array}{ll}
k & (i=d)\\
0 & (i\neq d)
\end{array}
\right..
\end{multline*}
\end{proof}

\begin{Example} There is a complex $\Bbb P$ of projective modules
over a noetherian
commutative ring $R$ such that for each $\frak m\in\Max(R)$, $\Bbb P(\frak m)$ 
is exact, but $\Bbb P$ is not exact, where $\Max(R)$ denotes the set of 
maximal ideals of $R$.
\end{Example}

\begin{proof} Let $R$ be a DVR with its field of fractions $K$, and $\Bbb P$ a
projective resolution of $K$.
\end{proof}

\end{document}